\documentclass[12pt]{amsart}
\usepackage{array,longtable}
\usepackage[utf8]{inputenc}
\usepackage{bbm}
\usepackage{amssymb,latexsym,fancyhdr,color,graphics}
\usepackage{amsmath, amsthm}
\usepackage{setspace}
\usepackage[dvips]{graphicx}
\usepackage[pdftex]{hyperref}
\usepackage[parfill]{parskip}
\newtheorem{lem}{\bf Lemma}

\newtheorem{prop}{\bf Proposition}
\newtheorem{thm}{\bf Theorem}
\newtheorem{corr}{\bf Corollary}
\newtheorem{ex}{\bf Example}
\newtheorem{rem}{\bf Remark}

\newcommand{\E}{\mbox{\bf E}}
\newcommand{\ind}{{\bf 1}}

\begin{document}

\title[On the Ruin Problem]{On the Ruin Problem with Investment when the Risky Asset is a Semimartingale}
\maketitle

\begin{center}
{\large  J. Spielmann, LAREMA, D\'epartement de
Math\'ematiques, Universit\'e d'Angers, 2, Bd Lavoisier  49045, \sc Angers Cedex 01}\vspace{0.5cm}

{\large  L. Vostrikova, LAREMA, D\'epartement de
Math\'ematiques, Universit\'e d'Angers, 2, Bd Lavoisier  49045,
\sc Angers Cedex 01}
\end{center} 
\vspace{0.2in}

\begin{abstract}
In this paper, we study the ruin problem with investment in a general framework where the business part $X$ is a L\'{e}vy process and the return on investment $R$ is a semimartingale. We obtain upper bounds on the finite and infinite time ruin probabilities that decrease as a power function when the initial capital increases. When $R$ is a L\'{e}vy process, we retrieve the well-known results. Then, we show that these bounds are asymptotically optimal in the finite time case, under some simple conditions on the characteristics of $X$. Finally, we obtain a condition for ruin with probability one when $X$ is a Brownian motion with negative drift and express it explicitly using the characteristics of $R$.
\end{abstract}


\begin{section}{Introduction and Main Results}\label{s1}

\par The estimation of the probability of ruin of insurance companies is a fundamental problem for market actors. In his seminal paper \cite{Cr}, Cram\'{e}r used a compound Poisson process with drift to model the value of an insurance company and showed that, under some assumptions on the parameters of the process, the probability of ruin decreases at least as an exponential function of the initial capital. Over time, the compound Poisson process has been replaced by more complex models. In a first generalisation, the value of the company is modeled by a L\'{e}vy process and then the ruin probability behaves essentially like the tail of the L\'{e}vy measure and, in the light-tailed case, this means that this probability decreases at least as an exponential function (see \cite{A}, \cite{Kl}, \cite{Ky}, and \cite{Spiel}). To generalise even further, it can be assumed that insurance companies invest their capital in a financial market. The main question is then: how does the probability of ruin changes with this additional source of risk?

In this general setting, the value of an insurance company with initial capital $y > 0$, denoted by $Y = (Y_t)_{t \geq 0}$, is given as the solution of the following linear stochastic differential equation
\begin{equation}\label{eq1}
Y_t = y + X_t + \int_0^t Y_{s-}dR_s, \text{ for all } t \geq 0,
\end{equation}
\noindent where $X = (X_t)_{t \geq 0}$ and $R = (R_t)_{t \geq 0}$ are two independent one dimensional stochastic processes defined on some probability space $(\Omega, \mathcal{F}, {\bf P})$ and chosen so that (\ref{eq1}) makes sense. In risk theory, the process $X$ represents the profit and loss of the business activity and $R$ represents the return of the investment. The main problem then concerns the study of the stopping time defined by
$$\tau(y) = \inf\{t \geq 0 | Y_t < 0\}$$
with $\inf\{\emptyset\}=+\infty$ and the evaluation of the ruin probability before time $T>0$, namely ${\bf P}(\tau(y) \leq T)$, and the ultimate ruin probability ${\bf P}(\tau(y) < +\infty)$. The ruin problem in this general setting was first studied in \cite{P3}.

Before describing our set-up and our results, we give a brief review of the relevant litterature. The special case when $R_t = rt$, with $r > 0$, for all $t \geq 0$ (non-risky investment) is well-studied and we refer to \cite{P1} and references therein for the main results. In brief, in that case and under some additional conditions, the ruin probability decreases even faster than an exponential since the capital of the insurance company is constantly increasing.

The case of risky investment is also well-studied. In that case, it is assumed in general that $X$ and $R$ are independent L\'{e}vy processes. The first results in this setting appear in \cite{K} (and later in \cite{Yu}) where it was shown that under some conditions there exists $C > 0$ and $y_0 \geq 0$ such that for all $y \geq y_0$ and for some $b > 0$
$${\bf P}(\tau(y) < +\infty) \geq Cy^{-b}.$$
Qualitatively, this means that the ruin probability cannot decrease faster as a power function, i.e. the degrowth is much slower than in the no-investment case. Later, under some conditions on the L\'{e}vy triplets of $X$ and $R$, it was shown in \cite{P} that for some $\beta>0$ and $\epsilon > 0$, there exists $C > 0$ such that, as $y \to \infty$,
$$y^{\beta}\,{\bf P}(\tau(y) < +\infty) = C + o(y^{-\epsilon}).$$
Recently, in \cite{KP1}, it is proven, under different assumptions on the L\'{e}vy triplets and when $X$ has no negative jumps, that there exists $C > 0$ such that  for the above $\beta>0$
$$\lim_{y \to \infty}y^{\beta}\,{\bf P}(\tau(y) < +\infty) = C.$$
Results concerning bounds on ${\bf P}(\tau(y) < +\infty)$ are given in \cite{K} where it is shown that, for all $\epsilon > 0$, there exists  $C > 0$ such that for all $y \geq 0$ and the same $\beta>0$
$${\bf P}(\tau(y) < +\infty) \leq C y^{-\beta + \epsilon}.$$ 

In less general settings similar results are available. The case when $X$ is a compound Poisson process with drift and exponential jumps and $R$ is a Brownian motion with drift is studied in \cite{F} (negative jumps only) and in \cite{KP2} (positive jumps only). In \cite{PZ} the model with negative jumps is generalized to the case where the drift of $X$ is a bounded stochastic process. 

Finally, some exact results for the ultimate ruin probability are available in specific models (see e.g. \cite{P1}, \cite{Yu}) and conditions for ruin with probability one are given, for different levels of generality, in \cite{F}, \cite{KP1}, \cite{KP2}, \cite{K}, \cite{P2} and \cite{PZ}.

The goal of this paper is to contribute to the study of the ruin problem by extending some results to the case when $R$ is a semimartingale and by obtaining similar results for the finite-time ruin probability in this general set-up. Thus, in the following we suppose that the processes $X = (X_t)_{t\geq 0}$ and $R=(R_t)_{t\geq 0}$ are independent one-dimensional processes both starting from zero, and such that $X$ is a L\'{e}vy process and $R$ is a semimartingale. We suppose additionally that the jumps of $R$ denoted $\Delta R_t = R_t - R_{t-}$ are strictly bigger than $-1$, for all $t > 0$. 

\par We denote the generating triplet of the L\'{e}vy process $X$ by $(a_X,\sigma^2_X,\nu_X)$ where $a_X\in \mathbb{R}$, $\sigma_X\geq 0$ and $\nu _X$ is a L\'{e}vy measure. We recall that the generating triplet characterizes the law of $X$ via the characteristic function $\phi_X$ of $X_t$ (see e.g. p.37 in \cite{Sa}):
\begin{equation*}\label{eq0}
\phi_X(\lambda ) = \exp\left(t \left(i\lambda a_X - \frac{\sigma ^2_X\lambda ^2 }{2} +
 \int_{\mathbb{R}}(e^{i\lambda x}-1-i\lambda x {\bf 1}_{\{|x| \leq 1 \}})\,\nu_X(dx)\right)\right)
\end{equation*}
where the Lévy measure $\nu_X$ satisfies
\begin{equation*}\label{c_0}
\int_{\mathbb{R}} \min(x^2, 1)\,\nu_X(dx)  < \infty.
\end{equation*}
As well-known, the process $X$ can then be written in the form:
\begin{equation}\label{decx}
\begin{split}
X_t= a_Xt & +\sigma_XW_t + \int_0^t\int_{|x|\leq 1}x(\mu_X(ds,dx)-\nu_X(dx)ds) \\
& + \int_0^t\int_{|x| > 1}x\mu_X(ds,dx),
\end{split}
\end{equation}
where $\mu_X$ is the measure of jumps of $X$ and $W$ is standard Brownian Motion.

We recall that  a semimartingale $R = (R_t)_{t\geq 0}$ can be also defined by its semimartingale decomposition, namely
\begin{equation}\label{eq_semdeco}
\begin{split}
R_t=B_t & + R^c_t + \int_0^t\int_{|x| \leq 1}x(\mu_R(ds,dx)-\nu_R(ds, dx)) \\
& + \int_0^t\int_{|x| > 1}x\mu_R(ds,dx),
\end{split}
\end{equation}
where $B=(B_t)_{t\geq 0}$ is a drift part, $R^c = (R^c_t)_{t\geq 0}$ is the continuous martingale part of $R$, $\mu_R$ is the measure of jumps of $R$ and $\nu_R$ is its compensator (see e.g. Chapter 2 of  \cite{JSh} for more information about these notions).

As well-known the equation \eqref{eq1} has a unique strong solution (see e.g. Theorem 11.3 in \cite{paulec}): for $t>0$
\begin{equation}\label{eq2}
Y_t= \mathcal{E}(R)_t\left(y+\int_0^t\frac{dX_s}{\mathcal{E}(R)_{s-}}\right)
\end{equation}
where $\mathcal{E}(R)$ is Doléans-Dade's exponential,
$$\mathcal{E}(R)_t= \exp\left(R_t-\frac{1}{2}\langle R^c\rangle _t\right)\prod_{0<s\leq t}(1+\Delta R_s)e^{-\Delta R_s}$$
(for more details about Doléans-Dade's exponential see e.g. Ch.1, §4f, p. 58 in \cite{JSh}).
Then the time of ruin is simply
\begin{equation}\label{eq3}
\tau (y)= \inf \left\{ t\geq 0 \left| \int_0^t\frac{dX_s}{\mathcal{E}(R)_{s-}}<-y\right.\right\}
\end{equation}
because $\mathcal{E}(R)_t >0$, for all $t\geq 0$, and this last fact follows from the assumption that $\Delta R_t > -1$, for all $t \geq 0$.

In this paper, we show that the behaviour of $\tau(y)$ for finite horizon $T > 0$ depends strongly on the behaviour of the exponential functionals at $T$, i.e. on the behaviour of
$$I_T= \int_0^{T}e^{-\hat{R}_{s}}ds \,\,\, \mbox{and}\,\,\,J_{T}(\alpha)= \int_0^{T}e^{-\alpha\hat{R}_{s}}ds$$
where $\alpha>0$ and  $\hat{R}_t = \ln \mathcal{E}(R)_t$, for all $t \geq 0$,
and for infinite horizon on the behaviour of
$$I_{\infty}= \int_0^{\infty}e^{-\hat{R}_{s}}ds \,\,\, \mbox{and}\,\,\,J_{\infty}(\alpha)= \int_0^{\infty}e^{-\alpha\hat{R}_{s}}ds.$$
For convenience we denote $J_T = J_T(2)$ and $J_{\infty} = J_{\infty}(2)$. More precisely, defining
$$\beta_T = \sup\left\{\beta \geq 0 : \E(J_T^{\beta/2}) < \infty, \E(J_T(\beta)) < \infty\right\},$$
we prove the following theorem.
\begin{thm}\label{t1}
Let $T > 0$. Assume that $\beta_T > 0$ and that, for some $0 < \alpha < \beta_T$, we have
\begin{equation}\label{intcond}
\int_{|x| > 1}|x|^{\alpha}\nu_X(dx) < \infty.
\end{equation}
Then, for all $y > 0$,
$${\bf P}(\tau(y) \leq T) \leq \frac{C_1\E(I_T^{\alpha}) + C_2\E(J_T^{\alpha/2}) + C_3 \E(J_T(\alpha))}{y^{\alpha}},$$
where the expectations on the right hand side are finite and $C_1 \geq 0$, $C_2 \geq 0$, and $C_3 \geq 0$ are constants that depend only on $\alpha$ in an explicit way.
\end{thm}

This theorem links the ruin probability with the tails of the L\'{e}vy measure of $X$ and the exponential functionals of the process $R$ which are well-studied objects. It also gives the first results for the case when $R$ belongs to the class of semimartingales, and the case when $R$ is a L\'{e}vy process is recovered as a special case. This could be used to study the ruin probabilities when the asset has stochastic volatility or when the investment is in a risk-free asset with a stochastic interest rate. Theorem \ref{t1} is also, up to our knowledge, the first result, when $R$ is not deterministic, for the ruin before a finite time for processes given by equations of the form (\ref{eq1}) even in the case when $R$ is a L\'{e}vy process.

From Theorem \ref{t1}, we can easily obtain a similar results for the ultimate ruin probability. Define
$$\beta_{\infty} = \sup\left\{\beta \geq 0 : \E(I_{\infty}^{\beta}) < \infty, \E(J_{\infty}^{\beta/2}) < \infty, \E(J_{\infty}(\beta)) < \infty\right\}.$$
Then, since $(I_t)_{t \geq 0}$, $(J_t)_{t \geq 0}$ and $(J_t(\alpha))_{t \geq 0}$ are increasing, we obtain, letting $T \to \infty$ and using the monotone convergence theorem with the upper bound of Theorem \ref{t1}, the following corollary.

\begin{corr}\label{c2}
Assume that $\beta_{\infty} > 0$ and that (\ref{intcond}) holds for some $0 < \alpha < \beta_{\infty}$, then
$${\bf P}(\tau(y) < \infty) \leq \frac{C_1\E(I_{\infty}^{\alpha}) + C_2\E(J_{\infty}^{\alpha/2}) + C_3 \E(J_{\infty}(\alpha))}{y^{\alpha}},$$
where $C_1 \geq 0$, $C_2 \geq 0$, and $C_3 \geq 0$ are constants that depend only on $\alpha$ in an explicit way.
\end{corr}

We can show, when $\beta_T \geq 1$ and under some simple conditions on the L\'evy triplet of $X$, that the bound in Theorem \ref{t1} is asymptotically optimal in a sense given below.

\begin{thm}\label{t2}
Let $T > 0$. Assume that $1 \leq \beta_T < \infty$ and that $\E(I_T^{\beta_T}) = +\infty$. Additionally, assume that 
\begin{equation*}
\int_{|x| > 1}|x| \nu_X(dx) < +\infty
\end{equation*}
and that 
\begin{equation}\label{eq_safety}
a_X + \int_{|x| > 1}x\nu_X(dx) < 0 \text{ or } \sigma_X > 0.
\end{equation}
Then, for all $\delta > 0$, there exists a positive numerical sequence $(y_n)_{n \in \mathbb{N}}$ increasing to $+\infty$ such that, for all $C > 0$, there exists $n_0 \in \mathbb{N}$ such that for all $n \geq n_0$,
$${\bf P}(\tau(y_n) \leq T) \geq \frac{C}{y_n^{\beta_T}\ln(y_n)^{1+\delta}}.$$

Moreover, if (\ref{intcond}) is satisfied for all $\alpha < \beta_T$, then,
$$\limsup_{y \to \infty}\frac{\ln\left({\bf P}(\tau(y) \leq T)\right)}{\ln(y)} = -\beta_T.$$
\end{thm}

To complete our study of the ruin problem in this setting, we give in our last result a sufficient condition for ruin with probability one in the particular case when $X$ is a Brownian motion with negative drift.

\begin{prop}\label{t3}
Assume that $X_t = a_X t + \sigma_X W_t$, for all $t \geq 0$, with $a_X \leq 0$, $\sigma_X \geq 0$ and $a_X^2 + \sigma_X > 0$. Assume also that $\lim_{t \to \infty}\frac{\hat{R}_t}{t} = \mu < 0$ $({\bf P}-a.s.)$. Then, for all $y > 0$,
$${\bf P}(\tau(y) < \infty) = 1.$$ 
\end{prop}

The rest of the paper is structured as follows. In Section 2, we point to the known results about exponential functionals of semimartingales, give a simple way to obtain $\beta_T$ and $\beta_{\infty}$ in the case when $R$ is a L\'{e}vy process and apply it to some examples. In Section 3, we prove Theorem \ref{t1} and, in Section 4, we prove Theorem \ref{t2}. In Section 5, we prove Theorem \ref{t3} and we obtain explicit conditions on the characteristics of $R$ to have $\lim_{t \to \infty}\frac{\hat{R}_t}{t} < 0$ $({\bf P}-a.s.)$. Finally, we show also that in the case when $R$ is a L\'evy process this corresponds to the known results.

\end{section}

\begin{section}{Exponential functionals of semimartingales}\label{s2}

Exponential functionals of semimartingales (especially of L\'evy processes) are very well-studied. The question of existence of the moments of $I_{\infty}$ and the formula in the case when $R$ is a subordinator was considered in \cite{BY}, \cite{CPY} and \cite{SV}. In the case when $R$ is a L\'evy process, the question of the existence of the density of the law of $I_{\infty}$, PDE equations for the density and the asymptotics for the law were investigated in \cite{Be}, \cite{BeL}, \cite{BJ}, \cite{BLM}, \cite{BS}, \cite{D}, \cite{EM}, \cite{GP}, \cite{KPS}, \cite{PRV}, \cite{PS} and \cite{R}. In the more general case of processes with independent increments, conditions for the existence of the moments and reccurent equations for the moments were studied in \cite{SV} and \cite{SV2}. The existence of the density of such functionals and the corresponding PDE equations were considered in \cite{V}. Here, we give two simple results concerning the finiteness of $\beta_T$ and $\beta_{\infty}$ when $R$ is a L\'evy process and apply them to the computation of $\beta_T$ and $\beta_{\infty}$ in some examples. Then, we present an example when $R$ is an additive process.

First of all, we give some basic facts about the exponential transform $\hat{R} = (\hat{R}_t)_{t\geq 0}$ of $R$, i.e. the process defined by
$$\mathcal{E}(R)_t= \exp(\hat{R}_t).$$
Since
$$\mathcal{E}(\hat{R}_t)= \exp\left( R_t-\frac{1}{2}\langle R^c\rangle_t + \sum_{0<s\leq t}(\ln(1+\Delta R_s)-\Delta R_s)\right)$$
we get that
$$\hat{R}_t= R_t-\frac{1}{2}\langle R^c\rangle_t + \sum_{0<s\leq t}(\ln(1+\Delta R_s)-\Delta R_s).$$
When $R$ is a semimartingale, the process $\hat{R}$ is also a semimartingale and the jumps of $\hat{R}$ are given by
$$\Delta \hat{R}_t = \ln(1+\Delta R_t), \text{ for all } t \geq 0.$$ 
Similarly, when $R$ is a L\'{e}vy process, the process $\hat{R}$ is also a L\'{e}vy process.

\begin{prop}\label{p0} Suppose that $R$ is a L\'{e}vy process. For $\alpha > 0$ and $T > 0$ the following conditions are equivalent:
\begin{enumerate}
\item[(i)] $\E(J_T(\alpha))<\infty,$
\item[(ii)] $\int_{|x| > 1}e^{-\alpha x}\nu_{\hat{R}}(dx) < \infty,$
\item[(iii)] $\int_{-1}^{\infty}\ind_{\{|\ln(1+x)| > 1\}}(1+x)^{-\alpha}\nu_{R}(dx) < \infty.$
\end{enumerate}
\end{prop}
\begin{proof}
By Fubini's theorem, we obtain
$$\E(J_T(\alpha)) = \E\left(\int_0^Te^{-\alpha\hat{R}_t}dt\right) = \int_0^T\E(e^{-\alpha \hat{R}_t})dt.$$
So, $\E(J_T(\alpha)) < \infty$ is equivalent to $\E(e^{-\alpha \hat{R}_t}) < \infty$, for all $t \geq 0$, which, by Theorem 25.3, p.159 in \cite{Sa}, is equivalent to 
$$\int_{|x| > 1}e^{-\alpha x}\nu_{\hat{R}}(dx) < \infty.$$
Then, note that
\begin{equation*}
\begin{split}
\int_{|x| > 1}e^{-\alpha x}\nu_{\hat{R}}(dx) & = \int_0^1\int_{|x| > 1}e^{-\alpha x}\nu_{\hat{R}}(dx)ds \\
& = \E\left(\sum_{0 < s \leq 1}e^{-\alpha \Delta \hat{R}_s}\ind_{\{|\Delta \hat{R}_s| > 1\}}\right) \\
& = \E\left(\sum_{0 < s \leq 1}(1+\Delta R_s)^{-\alpha}\ind_{\{|\ln(1+\Delta R_s)| > 1\}}\right) \\
& = \int_{-1}^{\infty}\ind_{\{|\ln(1+x)| > 1\}}(1+x)^{-\alpha}\nu_{R}(dx).
\end{split}
\end{equation*}
\end{proof}

Proposition \ref{p0} allows us to compute $\beta_T$ in some standard models of mathematical finance.

\begin{ex}\rm
Suppose that $\hat{R}$ is given by
$\hat{R}_t = a_{\hat{R}} t + \sigma_{\hat{R}} W_t + \sum_{n = 0}^{N_t}Y_n,$
where $a_{\hat{R}} \in \mathbb{R}$, $\sigma_{\hat{R}} \geq 0$, $W = (W_t)_{t \geq 0}$ is a standard Brownian motion and $N = (N_t)_{t \geq 0}$ is a Poisson process with rate $\gamma > 0$, and $(Y_n)_{n \in \mathbb{N}}$ is a sequence of iid random variables. Suppose, in addition, that all processes involved are independent. If for $(Y_n)_{n \in \mathbb{N}}$ we take any sequence of iid random variables with $\E(e^{-\alpha Y_1}) < \infty$, for all $\alpha > 0$, then $\beta_T = +\infty$. If for $(Y_n)_{n \in \mathbb{N}}$ we take a sequence of iid random variables with $\E(e^{-\alpha Y_1}) < \infty$, when $\alpha < \alpha_0$, for some $\alpha_0 > 0$, and $\E(e^{-\alpha_0 Y_1}) = +\infty$, then $\beta_T = \alpha_0$.
\end{ex}

\begin{ex}\rm
Suppose that $\hat{R}$ is a L\'{e}vy process with triplet $(a_{\hat{R}}, \sigma_{\hat{R}}^2, \nu_{\hat{R}})$, where $a_{\hat{R}} \in \mathbb{R}$, $\sigma_{\hat{R}} \geq 0$ and $\nu_{\hat{R}}$ is the measure on $\mathbb{R}$ given by
$$\nu_{\hat{R}}(dx) = \left(C_1 |x|^{-(1+\alpha_1)}e^{-\lambda_1|x|}\ind_{\{x < 0\}} + C_2 x^{-(1+\alpha_2)}e^{-\lambda_2 x}\ind_{\{x > 0\}}\right)dx,$$
where $C_1, C_2 > 0$, $\lambda_1, \lambda_2 > 0$ and $\alpha_1, \alpha_2 < 2$. This specification includes as special cases the Kou, CGMY and variance-gamma models (see e.g. Section 4.5 p.119 in \cite{cont}). We will show that if $\lambda_1 \geq 2$, then $\beta_T = \lambda_1$. Note that, using Proposition \ref{p0} and the change of variables $y = -x$, we see that $\E(J_T(\alpha)) < \infty$, for $\alpha > 0$, is equivalent to
$$C_1\int_{1}^{\infty}y^{-(1+\alpha_1)}e^{-(\lambda_1 - \alpha)y}dy + C_2 \int_{1}^{\infty}x^{-(1+\alpha_2)}e^{-(\alpha + \lambda_2) x}dx < \infty.$$
But, the first integral converges if $\alpha < \lambda_1$ and diverges if $\alpha > \lambda_1$ and second integral always converges. Now, if $\alpha \geq 2$, it is easy to show that $\E(J_T(\alpha)) < \infty$ implies $\E(J_T^{\alpha/2}) < \infty$ (see Lemma \ref{lem_csInt} below). Thus, if $\lambda_1 \geq 2$, we have $\beta_T = \lambda_1$.
\end{ex}

We now give an example when $R$ is not a L\'{e}vy process.
\begin{ex}\rm
Suppose that $L = (L_t)_{t \geq 0}$ is a L\'{e}vy process with triplet $(a_L, \sigma_L^2, \nu_L)$ where $\nu_L$ is assumed to be absolutely continuous w.r.t. the Lebesgue measure with density $f_L$. Suppose that $g$ is a deterministic, positive, measurable and square-integrable function on $\mathbb{R}$. Let $R_t = \int_0^tg(s-)dL_s$, for all $t \geq 0$. Then, in general, $R$ is a process with independent but non-homogeneous increments. From Proposition 1 and Example 3 in \cite{SV}, we see that, if $\alpha \geq 2$ and
\begin{equation}\label{eq_ex3}
\int_0^T\int_{x < -1}e^{-\alpha x g(s)}\nu_L(dx)ds < +\infty
\end{equation}
then $\E(J_T(\alpha)) < +\infty$ and, by Lemma \ref{lem_csInt} below, $\E(J_T^{\alpha/2}) < +\infty$. Thus, $\alpha \leq \beta_T$.
\end{ex}

\begin{prop}\label{p3} Suppose that the L\'evy process $\hat{R}$ admits a Laplace transform, for all $ t \geq 0$, i.e. for $\alpha >0$
$$\E(\exp(-\alpha \hat{R}_t))=\exp (t\psi_{\hat{R}} (\alpha))$$
and that its Laplace exponent $\psi_{\hat{R}}$ has a strictly positive root $\beta$. Then the following conditions are equivalent:
\begin{enumerate}
\item[(i)] $\E(I_{\infty}^{\alpha})<\infty$,
\item[(ii)] $\E(J_{\infty}^{\alpha/2})<\infty$,
\item[(iii)]$\E(J_{\infty}(\alpha))<\infty,$
\item[(iv)] $\alpha < \beta.$
\end{enumerate}
Therefore, $\beta_{\infty} = \beta$.
\end{prop}
\begin{proof}
Note that, for any $\alpha > 0$ and $k > 0$, 
\begin{equation*}
\begin{split}
\exp(t\psi_{\hat{R}}(\alpha)) & = \E(\exp(-\alpha \hat{R}_t)) = \E\left(\exp\left(-\frac{\alpha}{k} k\hat{R}_t\right)\right) \\
& = \exp\left(t\psi_{k\hat{R}}\left(\frac{\alpha}{k}\right)\right).
\end{split}
\end{equation*}
Therefore, $\psi_{\hat{R}}(\alpha) = \psi_{k\hat{R}}\left(\frac{\alpha}{k}\right)$, for all $\alpha > 0$ and $k > 0$. Then, Lemma 3 in \cite{R} yields the desired result.
\end{proof}

\begin{rem}
Note that the root of the Laplace exponent was already identified as the relevant quantity for the tails of ${\bf P}(\tau(y) < \infty)$ in \cite{P}.
\end{rem}

Using Proposition \ref{p3} we can compute $\beta_{\infty}$ in two important examples.

\begin{ex}\rm
Suppose that $R_t = a_{R}t +\sigma_R W_t$, for all $t \geq 0$, where $a_R \in \mathbb{R}$, $\sigma_R > 0$ and $W = (W_t)_{t \geq 0}$ is a standard Brownian motion, then $\hat{R}_t = \left(a_R - \frac{\sigma_R^2}{2}\right)t + \sigma_R W_t$, for all $t \geq 0$. Thus, we obtain $\psi_{\hat{R}}(\alpha) = -\left(a_R - \frac{1}{2}\sigma_R^2\right)\alpha + \frac{\sigma_R^2}{2}\alpha^2$ and, by Proposition \ref{p3}, $\beta_{\infty} = \frac{2a_R}{\sigma_R^2} - 1$. We remark that this coincides with the results in e.g. \cite{F} and \cite{KP2}.
\end{ex}

\begin{ex}\rm
Suppose that $\hat{R}_t = a_{\hat{R}}t + \sigma_{\hat{R}} W_t + \sum_{n = 0}^{N_t}Y_n$, where $a_{\hat{R}} \in \mathbb{R}$, $\sigma_{\hat{R}} \geq 0$ and $W = (W_t)_{t \geq 0}$ is a standard Brownian motion and $N = (N_t)_{t \geq 0}$ is a Poisson process with rate $\gamma > 0$, and $(Y_n)_{n \in \mathbb{N}}$ is a sequence of iid random variables with $\E(e^{-\alpha Y_1}) < \infty$, for all $\alpha > 0$. Suppose, in addition, that all processes involved are independent. It is easy to see that, for all $\alpha > 0$,
$$\psi_{\hat{R}}(\alpha) = -a_{\hat{R}}\alpha + \frac{\sigma_{\hat{R}}^2}{2}\alpha^2 + \gamma \left(\E(e^{-\alpha Y_1}) - 1\right).$$ 
Now, it is possible to show (see e.g. \cite{Spiel}) that the equation $\psi_{\hat{R}}(\alpha) = 0$ has an unique non-zero solution if, and only if, $\hat{R}$ is not a subordinator and $\psi'(0+) < 0$ which, under some additional conditions to invert the differentiation and expectation operators, is equivalent to $a_{\hat{R}} > \gamma \E(Y_1)$ (and which corresponds, in actuarial theory, to the "safety loading condition"). In that case, $\beta_{\infty}$ is the unique non-zero real solution of this equation.
\end{ex}

\end{section}

\begin{section}{Upper bound}\label{s3}

In this section, we prove Theorem \ref{t1}. We start with some preliminary results.

\begin{lem}\label{lem_csInt}
For all $T > 0$, we have the following.
\begin{enumerate}
\item[(a)] If $0 < \alpha < 2$, then $\E(J_T^{\alpha/2})<\infty$ implies $\E(I_T^{\alpha}) < \infty$ and $\E(J_T(\alpha)) < \infty$.
\item[(b)] If $\alpha \geq 2$, $\E(J_T(\alpha)) < \infty$ implies $\E(I_T^{\alpha}) < \infty$ and $\E(J_T^{\alpha/2}) < \infty$.
\end{enumerate}
\end{lem}

\begin{proof}
First note that by the Cauchy-Schwarz inequality we obtain, for all $T > 0$,
$$I_T = \int_0^T\mathcal{E}(R)_{s}^{-1}ds \leq \sqrt{T}\left(\int_0^T\mathcal{E}(R)_{s}^{-2}ds\right)^{1/2} = \sqrt{T}\sqrt{J_T}.$$
So, $\E(I_T^{\alpha}) \leq T^{\alpha/2}\E(J_T^{\alpha/2})$, for all $\alpha > 0$.

Now, if $0 < \alpha < 2$, we have $\frac{2}{\alpha} > 1$ and by H\"{o}lder's inequality
$$J_T(\alpha) = \int_0^T\mathcal{E}(R)^{-\alpha}_{s} ds \leq T^{(2-\alpha)/2}\left(\int_0^T\mathcal{E}(R)^{-2}_{s} ds\right)^{\alpha/2} = T^{(2-\alpha)/2}J_T^{\alpha/2}.$$
These inequalities yield (a).

Now, if $\alpha \geq 2$, we have either $\alpha = 2$ which yields the desired result or $\alpha > 2$. In that case, we have $\frac{\alpha}{2} > 1$ and, by H\"{o}lder's inequality, we obtain
\begin{equation*}
\begin{split}
J_T & = \int_0^T\mathcal{E}(R)^{-2}_{s} ds \leq T^{(\alpha-2)/\alpha}\left(\int_0^T\mathcal{E}(R)^{-\alpha}_{s} ds\right)^{2/\alpha} \\
& = T^{(\alpha-2)/\alpha}J_T(\alpha)^{2/\alpha}.
\end{split}
\end{equation*}
So, $\E(J_T^{\alpha/2}) \leq T^{(\alpha-2)/2}\E(J_T(\alpha))$, which yields (b).
\end{proof}

Denote by $M^d=(M^d_t)_{t \geq 0}$ the local martingale defined as:
$$M^d_t = \int_0^t\int_{|x|\leq 1}\frac{x}{\mathcal{E}(R)_{s-}}(\mu_X(ds,dx)-\nu_X(dx)ds)$$
and by $U = (U_t)_{t\geq 0}$ the process given by
$$U_t=\int_0^t\int_{|x|>1}\frac{x}{\mathcal{E}(R)_{s-}}\mu_X(ds,dx).$$
If $\int_{|x| > 1}|x|\nu_X(dx) < +\infty$, we can also define the local martingale $N^d = (N^d_t)_{t \geq 0}$ as
$$N^d_t = \int_0^t\int_{\mathbb{R}}\frac{x}{\mathcal{E}(R)_{s-}}(\mu_X(ds,dx)-\nu_X(dx)ds).$$

\begin{prop}\label{l2} We have the following identity in law:
\begin{equation*}\label{eq4}
\left(\int_0^t \frac{dX_s}{\mathcal{E}(R)_{s-}}\right)_{t\geq 0} \stackrel{\mathcal{L}}{=}
\left(a_X I_t + \sigma_X W_{J_t}+M^d_t+U_t\right) _{t\geq 0}.
\end{equation*}
Moreover, if $\int_{|x| > 1}|x|\nu_X(dx) < +\infty$, then,
\begin{equation*}
\left(\int_0^t \frac{dX_s}{\mathcal{E}(R)_{s-}}\right)_{t\geq 0} \stackrel{\mathcal{L}}{=}
\left(\delta_X I_t + \sigma_X W_{J_t} + N^d_t\right) _{t\geq 0},
\end{equation*}
where $\delta_X = a_X + \int_{|x| > 1}x\nu_X(dx)$.
\end{prop}

\begin{proof}
We show first that
$$\mathcal{L}\left(\left(\int_0^t \frac{dX_s}{\mathcal{E}(R)_{s-}}\right)_{t\geq 0}\,|\,\mathcal{E}(R)_s=q_s, s\geq 0 \right)= \mathcal{L}\left(\left(\int_0^t \frac{dX_s}{q_{s-}}\right)_{t\geq 0}\right)$$
\par To prove this equality in law we consider the representation of the stochastic integrals by Riemann sums (see
\cite{JSh}, Proposition I.4.44, p. 51). We recall that for any increasing sequence of stopping times $\tau=(T_n)_{n\in\mathbb{N}}$ with $T_0=0$ such that $\sup_n T_n=\infty$ and $T_n<T_{n+1}$ on the set $\{T_n<\infty\}$, Riemann approximation of the stochastic integral $\int_0^t \frac{dX_s}{\mathcal{E}(R)_{s-}}$ will be 
\begin{equation*}\label{eq4a}
\tau \left(\int_0^t \frac{dX_s}{\mathcal{E}(R)_{s-}}\right)= \sum_{n=0}^{\infty}\frac{1}{\mathcal{E}(R)_{T_n-}}\left( X_{T_{n+1}\wedge t} - X_{T_{n}\wedge t}\right)
\end{equation*}
The sequence $\tau_n= (T(n,m))_{m\in\mathbb{N}}$ of the adapted subdivisions is called Riemann sequence if $\sup_{m\in\mathbb{N}}(T(n,m+1)\wedge t-T(n,m)\wedge t)\rightarrow 0$ as $n\rightarrow \infty $ for all $t>0$. For our purposes we will take a deterministic Riemann sequences. Then, Proposition I.4.44, p.51 of \cite {JSh} says that for all $t>0$
\begin{equation}\label{eq5}
\tau_n \left(\int_0^t \frac{dX_s}{\mathcal{E}(R)_{s-}}\right)\stackrel{{\bf P}}{\longrightarrow}
\int_0^t \frac{dX_s}{\mathcal{E}(R)_{s-}}
\end{equation}
and
\begin{equation}\label{eq5a}
\tau_n \left(\int_0^t \frac{dX_s}{q_{s-}}\right)\stackrel{{\bf P}}{\longrightarrow}
\int_0^t \frac{dX_s}{q_{s-}}
\end{equation}
where $\stackrel{{\bf P}}{\longrightarrow}$ denotes the convergence in probability.
According to the Kolmogorov theorem, the law of the process is entirely defined by its finite-dimensional distributions. Let us take  for $k\geq 0$ a subdivision $t_0=0<t_1<t_2\cdots <t_k$
and a continuous bounded function $F:~\mathbb{R}^k\rightarrow~\mathbb{R}$, to prove by standard arguments that
$${\bf E}\left[ F\left(\tau_n\left(\int_0^{t_1} \frac{dX_s}{\mathcal{E}(R)_{s-}}\right), \cdots \tau_n\left(\int_0^{t_k} \frac{dX_s}{\mathcal{E}(R)_{s-}}\right)\right)\,|\,\mathcal{E}(R)_s=q_s, s\geq 0 \right]$$
$$= {\bf E}\left[ F\left(\tau_n\left(\int_0^{t_1} \frac{dX_s}{q_{s-}}\right), \cdots \tau_n\left(\int_0^{t_k} \frac{dX_s}{q_{s-}}\right) \right)\right]$$
Taking into account \eqref{eq5} and \eqref{eq5a}, we pass to the limit as $n\rightarrow \infty$ and we obtain
$${\bf E}\left[ F\left(\int_0^{t_1} \frac{dX_s}{\mathcal{E}(R)_{s-}}, \cdots \int_0^{t_k} \frac{dX_s}{\mathcal{E}(R)_{s-}}\right)\,|\,\mathcal{E}(R)_s=q_s, s\geq 0 \right]$$
$$\hspace{3cm}= {\bf E}\left[ F\left(\int_0^{t_1} \frac{dX_s}{q_{s-}}, \cdots \int_0^{t_k} \frac{dX_s}{q_{s-}}\right)\right]$$
and this proves the claim.
\par Using the decomposition \eqref{decx} we get that
\begin{equation*}
\begin{split}
\int_0^t \frac{dX_s}{q_{s-}} = a_X & \int_0^t \frac{ds}{q_{s}} + \sigma _X\int_0^t \frac{dW_s}{q_{s-}} \\
& + \int_0^t\int_{ |x| \leq 1}\frac{x}{q_{s-}} (\mu_X(ds,dx)-\nu _X(ds, dx)) \\
& + \int_0^t\int_{ |x|> 1}\frac{x}{q_{s-}} \,\mu_X(ds,dx).
\end{split}
\end{equation*}
We denote the last two terms in the r.h.s. of the equality above by $M_t^d(q)$ and $U_t(q)$ respectively. Recall that since $X$ is L\'{e}vy process the four processes appearing in the right-hand side of the above equality are independent. We use the well-known identity in law 
$$\left(\int_0^t \frac{dW_s}{q_{s-}} \right)_{t \geq 0}\stackrel{\mathcal{L}}{=}
\left(W_{\int_0^t \frac{ds}{q^2_{s}}} \right)_{t \geq 0}$$
to write
\begin{equation*}
\begin{split}
\left(a_X\int_0^t\frac{ds}{q_{s}}, \right. & \left. \sigma_X\int_0^t\frac{dW_s}{q_{s-}}, M^d_t(q),U_t(q)\right)_{t\geq0} \\
& \stackrel{\mathcal{L}}{=}\left(a_X\int_0^t\frac{ds}{q_{s}}, \sigma_XW_{\int_0^t\frac{ds}{q^2_{s}}}, M^d_t(q),U_t(q)\right)_{t\geq0}.
\end{split}
\end{equation*}
Then, we take the sum of these processes and we integrate w.r.t. the law of $\mathcal{E}(R)$. This yields the first result.

The proof of the second part is the same except we take the following decomposition of $X$:
$$X_t= \delta_X t+\sigma_X W_t+ \int_0^t\int_{\mathbb{R}}x(\mu_X(ds,dx)-\nu_X(dx)ds).$$
\end{proof}

The last ingredient in the proof of Theorem \ref{t1} are the Novikov maximal inequalities for compensated integrals with respect to random measures (see \cite{BJ}, \cite{Nov} and also \cite{mart}) which we will state below after introducing some notations. Let $f : (\omega, t, x) \mapsto f(\omega, t, x)$ be a left-continuous and measurable random function on $\Omega \times \mathbb{R_+} \times \mathbb{R}$. Specializing the notations of \cite{Nov} to our case, we say that $f \in F_2$ if, for almost all $\omega \in \Omega$,
$$\int_0^t\int_{\mathbb{R}}f(\omega, s, x)^2\nu_X(dx)ds < \infty.$$
If $f \in F_2$, we can define the compensated integral by
$$C_f(t) = \int_0^t\int_{\mathbb{R}}f(\omega, s, x)\left(\mu_X(ds, dx)-\nu_X(dx)ds\right)$$
for all $t \geq 0$. For these compensated integrals, we then have the following inequalities.
\begin{prop}[c.f. Theorem 1 in \cite{Nov}]\label{p8}
Let $f$ be a left-continuous measurable random function with $f \in F_2$. Let $C_f = (C_f(t))_{t \geq 0}$ be the compensated integral of $f$ as defined above. 
\begin{enumerate}
\item[(a)] For all $0 \leq \alpha \leq 2$,
$$\E\left(\sup_{0 \leq t \leq T}|C_f(t)|^{\alpha}\right) \leq K_1\E\left[\left(\int_0^T\int_{\mathbb{R}}f^2\nu_X(dx)ds\right)^{\alpha/2}\right].$$
\item[(b)] For all $\alpha \geq 2$,
\begin{equation*}
\begin{split}
\E\left(\sup_{0 \leq t \leq T}|C_f(t)|^{\alpha}\right) \leq K_2 & \E\left[\left(\int_0^T\int_{\mathbb{R}}|f|^2\nu_X(dx)ds\right)^{\alpha/2}\right] \\
& + K_3\E\left(\int_0^T\int_{\mathbb{R}}|f|^{\alpha}\nu_X(dx)ds\right)
\end{split}
\end{equation*}
\end{enumerate}
where $K_1 \geq 0$, $K_2 \geq 0$, and $K_3 \geq 0$ are constants depending only on $\alpha$ in an explicit way.
\end{prop} 

\begin{proof}[Proof of Theorem \ref{t1}]
Note that
$$\hspace{-50mm}\sup_{0 \leq t \leq T}-(a_X I_t + \sigma_X W_{J_t}+M^d_t+U_t)$$
$$\leq |a_X|I_T + \sup_{0 \leq t \leq T}\sigma_X|W_{J_t}| + \sup_{0 \leq t \leq T}|M^d_t| + \sup_{0 \leq t \leq T}|U_t|,$$
and that for positive random variable $Z_1, Z_2, Z_3, Z_4$ we have
$$\hspace{-70mm}\{Z_1 + Z_2 + Z_3 + Z_4 > y\}$$ 
$$ \subseteq \left\{Z_1 > \frac{y}{4}\right\} \cup \left\{Z_2 > \frac{y}{4}\right\} \cup \left\{Z_3 > \frac{y}{4}\right\} \cup \left\{Z_4 > \frac{y}{4}\right\}.$$
Therefore, using Proposition \ref{l2}, we obtain 
\begin{equation*}
\begin{split}
{\bf P}(\tau(y) \leq T) = & {\bf P}\left(\sup_{0 \leq t \leq T}-(a_X I_t + \sigma_X W_{J_t}+M^d_t+U_t) > y\right) \\
\leq & {\bf P}\left(|a_X| I_T > \frac{y}{4}\right) + {\bf P}\left(\sup_{0 \leq t \leq T}\sigma_X| W_{J_t}|> \frac{y}{4}\right) \\
& + {\bf P}\left(\sup_{0 \leq t \leq T}|M^d_t| > \frac{y}{4}\right) + {\bf P}\left(\sup_{0 \leq t \leq T}|U_t| > \frac{y}{4}\right).
\end{split}
\end{equation*}
For the first term, using Markov's inequality, we obtain
$${\bf P}\left(|a_X| I_T > \frac{y}{4}\right) \leq \frac{4^{\alpha}|a_X|^{\alpha}}{y^{\alpha}}\E(I_T^{\alpha}).$$
For the second term, since $(J_t)_{0 \leq t \leq T}$ is increasing we can change the time in the supremum and condition on $(\mathcal{E}({R})_t)_{0 \leq t \leq T}$ to obtain
\begin{equation*}
\begin{split}
{\bf P}\left(\sup_{0 \leq t \leq T}\sigma_X| W_{J_t}|> \frac{y}{4}\right) & = {\bf P}\left(\sup_{0 \leq t \leq J_T}\sigma_X|W_t|> \frac{y}{4}\right) \\
& =\E\left[{\bf P}\left(\left.\sup_{0 \leq t \leq J_T}\sigma_X|W_t|> \frac{y}{4} \right| (\mathcal{E}({R})_t)_{0 \leq t \leq T}\right)\right]
\end{split}
\end{equation*}
Since $W$ and $R$ are independent, we obtain, using the reflection principle, the fact that $W_{\int_0^Tq_t^{-2}dt} \overset{\mathcal{L}}{=} \left(\int_0^Tq_t^{-2}dt\right)^{1/2}W_1$ and Markov's inequality, that
\begin{equation*}
\begin{split}
{\bf P}&\left(\left.\sup_{0 \leq t \leq J_T}\sigma_X|W_t| > \frac{y}{4} \right| \mathcal{E}({R})_t = q_t, 0 \leq t \leq T\right) \\
& = 2{\bf P}\left(\left(\int_0^Tq_t^{-2}dt\right)^{1/2}\sigma_X|W_1| > \frac{y}{4}\right) \\
& \leq 2\frac{4^{\alpha}\sigma_X^{\alpha}}{y^{\alpha}}\left(\int_0^Tq_t^{-2}dt\right)^{\alpha/2}\E(|W_1|^{\alpha}).
\end{split}
\end{equation*}
Then, since $\E(|W_1|^{\alpha}) = \frac{2^{\alpha/2}}{\sqrt{\pi}}\Gamma\left(\frac{\alpha + 1}{2}\right)$, we obtain
$${\bf P}\left(\sup_{0 \leq t \leq T}\sigma_X| W_{J_t}|> \frac{y}{4}\right) \leq \frac{2^{(5\alpha+2)/2}\Gamma\left(\frac{\alpha + 1}{2}\right)\sigma_X^{\alpha}}{\sqrt{\pi}y^{\alpha}}\E(J_T^{\alpha/2}).$$
Note that the inequalities for the first two terms work for all $\alpha > 0$.

\par Suppose now that $0 < \alpha \leq 1$. We see that $\mathcal{E}(R)^{-1}_{t-}(\omega)x\ind_{\{|x| \leq 1\}} \in F_2$. Therefore, using Markov's inequality and part (a) of Proposition \ref{p8}, we obtain
\begin{equation*}
\begin{split}
& {\bf P}\left(\sup_{0 \leq t \leq T}|M^d_t| > \frac{y}{4}\right) \leq \frac{4^{\alpha}}{y^{\alpha}}\E\left(\sup_{0 \leq t \leq T}|M^d_t|^{\alpha}\right) \\
& \leq K_1 \frac{4^{\alpha}}{y^{\alpha}}\E\left[\left(\int_0^T\int_{\mathbb{R}}\frac{x^2}{\mathcal{E}(R)^2_{s-}}\ind_{\{|x| \leq 1\}}\nu_X(dx)ds\right)^{\alpha/2}\right] \\
& = K_1 \frac{4^{\alpha}}{y^{\alpha}}\left(\int_{\mathbb{R}}x^2\ind_{\{|x| \leq 1\}}\nu_X(dx)\right)^{\alpha/2}\E(J_T^{\alpha/2}).
\end{split}
\end{equation*}
For the last term, note that since $0 < \alpha \leq 1$, we have $\left(\sum_{i = 1}^Nx_i\right)^{\alpha} \leq \sum_{i = 1}^Nx_i^{\alpha}$, for $x_i \geq 0$ and $N \in \mathbb{N}^*$ and, for each $t \geq 0$,
\begin{equation*}
\begin{split}
|U_t|^{\alpha} & \leq \left(\sum_{0 < s \leq t}\mathcal{E}(R)^{-1}_{s-}|\Delta X_s|\ind_{\{|\Delta X_s| > 1\}}\right)^{\alpha} \\
& \leq \sum_{0 < s \leq t}\mathcal{E}(R)^{-\alpha}_{s-}|\Delta X_s|^{\alpha}\ind_{\{|\Delta X_s| > 1\}} \\
& = \int_0^t\int_{\mathbb{R}}\mathcal{E}(R)^{-\alpha}_{s-}|x|^{\alpha}\ind_{\{|x| > 1\}}\mu_X(ds,dx).
\end{split}
\end{equation*}
Therefore, using Markov's inequality and the compensation formula (see e.g. Theorem II.1.8 p.66-67 in \cite{JSh}), we obtain
\begin{equation*}
\begin{split}
{\bf P}\left(\sup_{0 \leq t \leq T}|U_t|\right. & \left. > \frac{y}{4}\right) \leq \frac{4^{\alpha}}{y^{\alpha}}\E\left(\sup_{0 \leq t \leq T}|U_t|^{\alpha}\right) \\
& \leq \frac{4^{\alpha}}{y^{\alpha}}\E\left(\sup_{0 \leq t \leq T}\int_0^t\int_{\mathbb{R}}\mathcal{E}(R)^{-\alpha}_{s-}|x|^{\alpha}\ind_{\{|x| > 1\}}\mu_X(ds,dx)\right) \\
& = \frac{4^{\alpha}}{y^{\alpha}}\E\left(\int_0^T\int_{\mathbb{R}}\mathcal{E}(R)^{-\alpha}_{s-}|x|^{\alpha}\ind_{\{|x| > 1\}}\nu_X(dx)ds\right) \\
& = \frac{4^{\alpha}}{y^{\alpha}}\left(\int_{\mathbb{R}}|x|^{\alpha}\ind_{\{|x| > 1\}}\nu_X(dx)\right)\E(J_T(\alpha)).
\end{split}
\end{equation*}
This finishes the proof when $0 < \alpha \leq 1$.

\par Suppose now that $1 < \alpha \leq 2$. The bound for ${\bf P}\left(\sup_{0 \leq t \leq T}|M^d_t| > \frac{y}{4}\right)$ can be obtained in the same way as in the previous case. Applying H\"{o}lder's inequality we obtain
\begin{equation*}
\begin{split}
|U_t|^{\alpha} \leq & \left(\int_0^t\int_{\mathbb{R}}\mathcal{E}(R)^{-1/\alpha}_{s-}\mathcal{E}(R)^{1/\alpha-1}_{s-}|x|\ind_{\{|x| > 1\}}\mu_X(ds, dx)\right)^{\alpha} \\
\leq & \left(\int_0^t\int_{\mathbb{R}}\mathcal{E}(R)^{-1}_{s-}|x|^{\alpha}\ind_{\{|x| > 1\}}\mu_X(ds, dx)\right) \\
& \times \left(\int_0^t\int_{\mathbb{R}}\mathcal{E}(R)^{-1}_{s-}\ind_{\{|x| > 1\}}\mu_X(ds, dx)\right)^{\alpha-1} \\
\leq & \left(\int_0^t\int_{\mathbb{R}}\mathcal{E}(R)^{-1}_{s-}|x|^{\alpha}\ind_{\{|x| > 1\}}\mu_X(ds, dx)\right)^{\alpha}.
\end{split}
\end{equation*}
Then, using Markov's inequality and the compensation formula, we obtain
\begin{equation*}
\begin{split}
{\bf P}\left(\sup_{0 \leq t \leq T}|U_t| > \frac{y}{4}\right) & \leq \frac{4^{\alpha}}{y^{\alpha}}\E\left(\sup_{0 \leq t \leq T}|U_t|^{\alpha}\right) \\
& = \left(\int_{\mathbb{R}}|x|^{\alpha}\ind_{\{|x| > 1\}}\nu_X(dx)\right)^{\!\alpha}\E(I_T^{\alpha}).
\end{split}
\end{equation*}
This finishes the proof in the case $1 < \alpha \leq 2$.

Finally, suppose that $\alpha \geq 2$. The estimation for ${\bf P}\left(\sup_{0 \leq t \leq T}|U_t| > \frac{y}{4}\right)$ still works in this case. Moreover, since $\mathcal{E}(R)^{-1}_{t-}(\omega)x\ind_{\{|x| \leq 1\}} \in F_2$, we obtain, applying part (b) of Proposition \ref{p8} that
\begin{equation*}
\begin{split}
{\bf P}\left(\sup_{0 \leq t \leq T}|M^d_t| > \frac{y}{4}\right) \leq & K_2 \E\left[\left(\int_0^T\int_{\mathbb{R}}\mathcal{E}(R)^{-2}_{s-}x^2\ind_{\{|x| \leq 1\}}\nu_X(dx)ds\right)^{\alpha/2}\right] \\
& + K_3\E\left(\int_0^T\int_{\mathbb{R}}\mathcal{E}(R)^{-\alpha}_{s-}|x|^{\alpha}\ind_{\{|x| \leq 1\}}\nu_X(dx)ds\right) \\
= & K_2 \left(\int_{\mathbb{R}}x^2\ind_{\{|x| \leq 1\}}\nu_X(dx)\right)^{\alpha/2}\E(J_T^{\alpha/2}) \\
& + K_3 \left(\int_{\mathbb{R}}|x|^{\alpha}\ind_{\{|x| \leq 1\}}\nu_X(dx)\right)\E(J_T(\alpha)).
\end{split}
\end{equation*}
Note that the right-hand side is finite since $|x|^{\alpha}\ind_{\{|x| \leq 1\}} \leq |x|^2\ind_{\{|x| \leq 1\}}$ when $\alpha \geq 2$. This finishes the proof.
\end{proof} 

\end{section}

\begin{section}{Asymptotic lower bound}

In this section, we prove Theorem 2 and, therefore, show that the upper bound obtained in Theorem \ref{t1} is asymptotically optimal for a large class of L\'{e}vy processes $X$. We start with some preliminary results. Denote $x^{+,p} = (\max(x,0))^{p}$, for all $x \in \mathbb{R}$ and $p > 0$.

\begin{lem}\label{lem_aysmZ}
Suppose that a random variable $Z > 0$ $({\bf P}-a.s.)$ satisfies $\E(Z^p) = \infty$, for some $p > 0$. Then, for all $\delta > 0$, there exists a positive numerical sequence $(y_n)_{n \in \mathbb{N}}$ increasing to $+\infty$ such that, for all $C > 0$, there exists $n_0 \in \mathbb{N}$ such that for all $n \geq n_0$,
$${\bf P}(Z \geq y_n) \geq \frac{C}{y_n^{p}\ln(y_n)^{1+\delta}}.$$ 
\end{lem}

\begin{proof}
If $Z > 0$ $({\bf P}-a.s.)$ is a random variable and $g:\mathbb{R_+} \to \mathbb{R_+}$ is a function of class $C^1$ with positive derivative, then, using Fubini's theorem, we obtain
$$g(0) + \int_0^{\infty}g'(u) {\bf P}(Z \geq u) du = g(0) + \E\left(\int_0^Zg'(u)du\right) = \E(g(Z)).$$ 
Applying this to the function $g(z) = z^p$ with $p>0$ we obtain, for all $y \geq e$,
$$\int_y^{\infty}u^{p-1} {\bf P}(Z \geq u) du = \infty.$$
Moreover, for all $\delta > 0$,
$$\sup_{u \geq y}[u^p\ln(u)^{1+\delta}{\bf P}(Z \geq u)]\int_y^{\infty}\frac{du}{u\ln(u)^{1+\delta}} \geq \int_y^{\infty}u^{p-1} {\bf P}(Z \geq u) du.$$
So, since $\int_y^{\infty}\frac{du}{u\ln(u)^{1+\delta}} < \infty$, we obtain, for all $y \geq e$,
$$\sup_{u \geq y}[u^p\ln(u)^{1+\delta}{\bf P}(Z \geq u)] = \infty.$$
Therefore, there exists a numerical sequence $(y_n)_{n \in \mathbb{N}}$ increasing to $+\infty$ such that,
$$\lim_{n \to \infty}y_n^p\ln(y_n)^{1+\delta}{\bf P}(Z \geq y_n) = +\infty.$$
\end{proof}

\begin{lem}\label{lem_sumXY}
Assume that $X$ and $Y$ are independent random variables with $\E(Y) = 0$. Assume that $p \geq 1$. Then, $\E[X^{+,p}] \leq \E[(X+Y)^{+,p}]$.
\end{lem}

\begin{proof}
For each $x \in \mathbb{R}$, we define the function $h_x : y \mapsto (x+y)^{+,p}$ on $\mathbb{R}$. Since $p \geq 1$, $h_x$ is a convex function and we obtain, using Jensen's inequality, that for each $x \in \mathbb{R}$,
$$\E[(x + Y)^{+,p}] = \E[h_x(Y)] \geq h_x(\E(Y)) = h_x(0) = x^{+,p}.$$
We obtain the desired result by integrating w.r.t. the law of $X$.
\end{proof}

\begin{lem}\label{lem_lowB}
Let $T > 0$. Assume that $a < 0$ or $\sigma > 0$ and that there exists $\beta > 0$ such that $\E(I_T^{\beta}) = \infty$. Then, $\E[(-a I_T - \sigma W_{J_T})^{+, \beta}] = \infty$.
\end{lem}

\begin{proof}
Suppose first that $a < 0$ and $\sigma = 0$. Then, 
$$\E[(-a I_T - \sigma W_{J_T})^{+, \beta}] = |a|^{\beta}\E(I_T^{\beta}) = \infty.$$

Next, suppose that $a \leq 0$ and $\sigma > 0$. In that case, using the  identities in law $W\overset{\mathcal{L}}{=}-W$  and $W_{J_T} \overset{\mathcal{L}}{=} \sqrt{J_T}W_1$, the Cauchy-Schwarz inequality and the conditional independence  between $W_1$ and $J_T$ given $\mathcal{E}(R)$, we obtain
\begin{equation*}
\begin{split}
\E[(-a I_T - \sigma W_{J_T})^{+, \beta}] & \geq \E[(\sigma\sqrt{J_T}W_1)^{+,\beta}] = \sigma^{\beta}\E(W_1^{+,\beta})\E(J_T^{\beta/2}) \\
& \geq \sigma^{\alpha}\E(W_1^{+,\beta})T^{-\beta/2}\E(I_T^{\beta}) = \infty.
\end{split}
\end{equation*}

Finally, if $a > 0$ and $\sigma > 0$, using the fact that $W\overset{\mathcal{L}}{=}-W$, that $W_{J_T} \overset{\mathcal{L}}{=} \sqrt{J_T}W_1$ and choosing $C > 1$, we obtain that
\begin{equation*}
\begin{split}
\E[(-a I_T - \sigma W_{J_T})^{+, \beta}] & = \E[(-|a| I_T + \sigma \sqrt{J_T} W_1)^{+,\beta}] \\
& \geq \E[(-|a| I_T + \sigma \sqrt{J_T} W_1)^{+,\beta}\ind_{\{\sigma \sqrt{J_T}W_1 \geq C|a|I_T\}}] \\
& \geq \E[((C-1)|a|I_T)^{\beta}\ind_{\{\sigma \sqrt{J_T}W_1 \geq C|a|I_T\}}].
\end{split}
\end{equation*}
Since $\frac{I_T}{\sqrt{J_T}} \leq \sqrt{T}$, by Cauchy-Schwarz's inequality, we obtain using the conditional independence  between $W_1$ and $I_T$ given $\mathcal{E}(R)$
\begin{equation*}
\begin{split}
\E[(-a I_T - \sigma W_{J_T})^{+, \beta}] & \geq \E\left[((C-1)|a|I_T)^{\beta}\ind_{\left\{W_1 \geq \frac{C|a|\sqrt{T}}{\sigma }\right\}}\right] \\
& = {\bf P}\left(W_1 \geq \frac{C|a|\sqrt{T}}{\sigma }\right)(C-1)^{\beta}|a|^{\beta}\E(I_T^{\beta}) = \infty.
\end{split}
\end{equation*}
\end{proof}

\begin{proof}[Proof of Theorem \ref{t2}]
The assumptions imply  $\int_{|x| > 1}|x|\nu_X(dx) <~+\infty$ and so, by Proposition \ref{l2}, we obtain
$${\bf P}\left(\sup_{0 \leq t \leq T}\left(-\int_0^t \frac{dX_s}{\mathcal{E}(R)_{s-}}\right) \geq y\right) \geq {\bf P}((-\delta_X I_T - \sigma_X W_{J_T} - N^d_T)^{+} \geq y),$$
where $\delta_X$ and $N^d = (N^d_t)_{t \in [0, T]}$ are defined as in Proposition \ref{l2}.

Then, by independence, we obtain
\begin{equation*}
\begin{split}
\E[(&-\delta_X I_T - \sigma_X W_{J_T} - N^d_T)^{+,\beta_T}] \\
& = \int_D\E[(-\delta_X I_T(q) - \sigma_X W_{J_T(q)} - N^d_T(q))^{+,\beta_T}]{\bf P}(\mathcal{E}(R) \in dq),
\end{split}
\end{equation*}
where $D$ is the Skorokhod space of c\`{a}dl\`{a}g functions on $[0,T]$, the measure ${\bf P}(\mathcal{E}(R) \in dq)$ is the law of $(\mathcal{E}(R)_t)_{t \in [0,T]}$, $I_T(q) = \int_0^T\frac{ds}{q_s}$, $J_T(q) = \int_0^T\frac{ds}{q^2_s}$ and
\begin{equation*}
\begin{split}
N^d_T(q) = \int_0^T\int_{|x| \leq 1}&\frac{x}{q_{s-}}(\mu_X(ds,dx)-\nu_X(dx)ds) \\
& + \int_0^T\int_{|x| > 1}\frac{x}{q_{s-}}(\mu_X(ds,dx)-\nu_X(dx)ds).
\end{split}
\end{equation*}

Denote by $N'_T(q)$ and $N''_T(q)$ the two terms on the r.h.s. of the equation above. Fixing $q \in D$, we now prove that $\E(N'_T(q)) = 0$ and $\E(N''_T(q)) = 0$. First, note that by Theorem 1 p.176 in \cite{LSh} and Theorem II.1.8 p.66-67 in \cite{JSh}, we find that
\begin{equation*}
\begin{split}
\E([N'_.(q), N'_.(q)]_T) & = \E\left(\int_0^T\int_{|x| \leq 1}\frac{x^2}{q_{s-}^2}\mu_X(ds, dx)\right) \\
& = \E\left(\int_0^T\int_{|x| \leq 1}\frac{x^2}{q_s^2}\nu_X(dx)ds\right) \\
& = \left(\int_0^T\frac{ds}{q^2_s}\right)\left(\int_{|x| \leq 1}x^2\nu_X(dx)\right).
\end{split}
\end{equation*}
Then, since $q$ a strictly positive c\`adl\`ag function on a compact interval, it is bounded with $\int_0^T\frac{ds}{q^2_s} < +\infty$ and since $\int_{|x| \leq 1}x^2\nu_X(dx) < +\infty$ by definition of the L\'{e}vy measure, we have $\E([N'_.(q),N'_.(q)]_T) < +\infty$. This shows that $N'(q)$ is a (square integrable) martingale and so $\E(N'_T(q)) = 0$.
For the second term, similarly we have
$$\int_0^T\int_{|x| > 1}\frac{|x|}{q_s}\nu_X(dx)ds = \left(\int_0^T\frac{ds}{q_s}\right)\left(\int_{|x| > 1}|x|\nu_X(dx)\right) < +\infty.$$
Therefore, by Proposition II.1.28 p.72 in \cite{JSh} and Theorem II.1.8 p.66-67 in \cite{JSh}, we have
\begin{equation*}
\begin{split}
\E(N''_T(q)) = \E\left(\int_0^T\right. & \left.\int_{|x| > 1}\frac{x}{q_{s-}}\mu_X(ds, dx)\right) \\
& - \E\left(\int_0^T\int_{|x| > 1}\frac{x}{q_{s-}}\nu_X(dx)ds\right) = 0.
\end{split}
\end{equation*}

Now, since the random variables $-\delta_X I_T(q) - \sigma_X W_{J_T(q)}$ and $-N^d_T(q)$ are independent and $\E(N_T^d(q)) = 0$, for all $q \in D$, we can apply Lemma \ref{lem_sumXY} to obtain 
$$\E[(-\delta_X I_T - \sigma_X W_{J_T} - N^d_T)^{+,\beta_T}] \geq \E[(-\delta_X I_T - \sigma_X W_{J_T})^{+,\beta_T}].$$
Then, using Lemma \ref{lem_aysmZ} and Lemma \ref{lem_lowB} with $a = \delta_X$ and $\sigma = \sigma_X$, we can conclude that for all $\delta > 0$, there exists a strictly positive sequence $(y_n)_{n \in \mathbb{N}}$ increasing to $+\infty$ such that, for all $C > 0$, there exists $n_0 \in \mathbb{N}$ such that, for all $n \geq n_0$,
$${\bf P}(\tau(y_n) \leq T) \geq \frac{C}{y_n^{\beta_T}\ln(y_n)^{1+\delta}}.$$

For the second part, note that the above implies that
$$\limsup_{y \to \infty}\frac{\ln\left({\bf P}(\tau(y) \leq T)\right)}{\ln(y)} \geq -\beta_T + \lim_{n \to \infty}\frac{\ln(C) - \ln(\ln(y_n)^{1+\delta})}{\ln(y_n)} = -\beta_T.$$
Now, using Theorem \ref{t1}, we obtain 
$$\limsup_{y \to \infty}\frac{\ln({\bf P}(\tau(y) \leq T))}{\ln(y)} \leq -\alpha,$$
for all $\alpha < \beta_T$, and letting $\alpha \to \beta_T$, we obtain
$$\limsup_{y \to \infty}\frac{\ln\left({\bf P}(\tau(y) \leq T)\right)}{\ln(y)} \leq -\beta_T,$$
and, hence, the claimed equality.
\end{proof}
\end{section}

\begin{section}{Conditions for Ruin with Probability 1}\label{s0}

In this section, after giving a simple result about the limits of the exponential functionals, we prove Theorem \ref{t3}. Then, we state explicit condition on the characteristics of $R$ for ruin with probability one and apply it to the L\'evy case.

\begin{lem}\label{lem_limfracR}
Assume that $\lim_{t \to \infty}\frac{\hat{R}_t}{t} = \mu < 0$ $({\bf P}-a.s.)$. Then,
$$\lim_{t \to \infty}I_t = +\infty \text{ and } \lim_{t \to \infty}J_t = +\infty  \hspace{2mm} ({\bf P}-a.s.).$$
\end{lem}

\begin{proof}
Since $\lim_{t \to \infty}\frac{\hat{R}_t}{t} = \mu < 0$ ({\bf P}-a.s.) implies that $\lim_{t \to \infty}\hat{R}_t = -\infty$ ({\bf P}-a.s.), we can show that $I_t = \int_0^t e^{-\hat{R}_s}ds$ and $J_t = \int_0^t e^{-2\hat{R}_s}ds$ diverge ({\bf P}-a.s.). In fact, denote by $\Omega_0$ a set of probability one such that $\lim_{t \to \infty}\hat{R}_t(\omega) = -\infty$, for each $\omega \in \Omega_0$, i.e. for each $C > 0$, there exists $t_0(\omega) \geq 0$, such that, for all $t \geq t_0(\omega)$, $-\hat{R}_t(\omega) \geq C$. Then, for each $\omega \in \Omega_0$ and for each $K > 0$, we have, taking $C = \ln(K + 1)$ and $\tau(\omega) = t_0(\omega) + 1$, that, for all $t \geq \tau(\omega)$,
$$\int_0^t e^{-\hat{R}_s(\omega)}ds \geq \int_{t_0(\omega)}^{\tau(\omega)} e^{-\hat{R}_s(\omega)}ds \geq e^c = K+1 \geq K.$$
The proof of the divergence for $J_t$ is similar.
\end{proof}

\begin{proof}[Proof of Proposition \ref{t3}]
Using Proposition \ref{l2}, we have, for all $y > 0$,
\begin{equation*}
\begin{split}
{\bf P}(\tau(y) < \infty) & = {\bf P}\left(\sup_{t \geq 0}(-a_X I_t - \sigma_X W_{J_t}) \geq y\right) \\
& \geq {\bf P}\left(\limsup_{t \to \infty}(-a_X I_t - \sigma_X W_{J_t}) \geq y\right).
\end{split}
\end{equation*}

When $\sigma_X = 0$, we have by assumption that $a_X < 0$, and therefore 
$${\bf P}\left(\limsup_{t \to \infty} (-a_X I_t - \sigma_X W_{J_t}) \geq y\right) = {\bf P}\left(\limsup_{t \to \infty} |a_X| I_t \geq y\right) = 1.$$

When $\sigma_X > 0$, since $W$ is a Brownian motion and $\lim_{t \to \infty} J_t = +\infty$, we have $\limsup_{t \to \infty}W_{J_t} = + \infty$ and thus
$${\bf P}\left(\limsup_{t \to \infty} (-a_X I_t - \sigma_X W_{J_t}) \geq y\right) = 1.$$
\end{proof}

Under some integrability conditions, we can prove a more explicit condition for ruin with probability one.

\begin{prop}\label{prop_ruin1}
Assume that $X_t = a_X t + \sigma_X W_t$, for all $t \geq 0$, with $a_X \leq 0$, $\sigma_X \geq 0$ and $a_X^2 + \sigma_X > 0$. Assume that 
\begin{enumerate}
\item[(i)] $\int_0^{\infty}(1+s)^{-2}d\langle R^c \rangle_s < \infty$,
\item[(ii)] there exists $p \in (1,2)$ such that
$$\int_0^{\infty}\int_{-1}^{\infty}\frac{\min(|\ln(1+x)|^2,|\ln(1+x)|^p)}{(1+s)^p} \nu_R(ds, dx) < \infty,$$
\item[(iii)]there exists $D<0$ such that $({\bf P}-a.s.)$,
$$\hspace{-6cm}D=\lim_{t \to \infty}\frac{1}{t}\left(B_t - \frac{1}{2}\langle R^c \rangle_t \right.$$
$$\hspace{3cm}\left. + \int_0^t\int_{-1}^{\infty}\left(\ln(1+x) - x \ind_{\{|\ln(1+x)| \leq 1\}}\right)\nu_R(ds, dx)\right)$$
where $B = (B_t)_{t \geq 0}$ is the drift part of $R$.
\end{enumerate}
Then, for all $y > 0$,
$${\bf P}(\tau(y) < \infty) = 1.$$ 
\end{prop}

\begin{proof}
We are going to show that $\lim_{t \to \infty}\frac{\hat{R}_t}{t} = D$. Since, $s \mapsto (1+s)^{-p}$ is a continuous function, for each $t > 0$, we have $(1+s)^{-p} \geq d_t$ for some constant $d_t > 0$ and for all $s \in [0, t]$. Thus, we have, for all $t \geq 0$,
\begin{equation*}
\begin{split}
\int_0^t \int_{-1}^{\infty}&|\ln(1+x)| \ind_{\{|\ln(1+x)| > 1\}} \nu_R(ds, dx) \\
& \leq \frac{1}{d_t}\int_0^t\int_{-1}^{\infty}\frac{|\ln(1+x)|^p}{(1+s)^p} \ind_{\{|\ln(1+x)| > 1\}} \nu_R(ds, dx)< \infty.
\end{split}
\end{equation*}
Thus, using the semimartingale decomposition of $R$ for the truncation function $h(x) = \ind_{\{|\ln(1+x)| > 1\}}$ and Proposition II.1.28 p.72 in \cite{JSh}, we obtain
\begin{equation*}
\begin{split}
\hat{R}_t = B_t & - \frac{1}{2} \langle R^c \rangle_t + \int_0^t \int_{-1}^{\infty}\left(\ln(1+x) - x\ind_{\{|\ln(1+x)| \leq 1\}}\right)\nu_R(ds, dx) \\
& + R^c_t + \int_0^t\int_{-1}^{\infty} \ln(1+x)\ind_{\{|\ln(1+x)| \leq 1\}} (\mu_R(ds, dx) - \nu_R(ds, dx)) \\
& + \int_0^t\int_{-1}^{\infty} \ln(1+x)\ind_{\{|\ln(1+x)| > 1\}} (\mu_R(ds, dx) - \nu)_R(ds, dx)).
\end{split}
\end{equation*}

Denoting by $H'_t$ and $H''_t$ the last two terms of the r.h.s. of the equation above, we show that $\lim_{t \to \infty}\frac{R^c_t}{t} = 0$, $\lim_{t \to \infty}\frac{H'_t}{t} = 0$, and $\lim_{t \to \infty}\frac{H''_t}{t} = 0$ $({\bf P}-a.s.)$. 

For $H'$ and $H''$ we apply Theorem 9 p.142-143 in \cite{LSh}. Since $H'$ is purely discontinuous, this theorem tells us that $\lim_{t \to \infty}\frac{H'_t}{t} = 0$ $({\bf P}-a.s.)$, if $\tilde{Q}_{\infty} < +\infty$, where $\tilde{Q}$ is the compensator of the process $Q = (Q_t)_{t \geq 0}$ given by
$$Q_t = \sum_{0 < s \leq t} \frac{(\Delta H'_s / (1+s))^2}{1 + |\Delta H'_s / (1+s)|}.$$
The same holds for $H''$ when we replace $\Delta H'_t$ by $\Delta H''_t$.

Since $\Delta H'_t = \ln(1 + \Delta R_t) \ind_{\{|\ln(1 + \Delta R_t)| \leq 1\}}$ and $p \leq 2$, we have
\begin{equation*}
\begin{split}
\tilde{Q}_{\infty} & = \int_0^{\infty}\int_{-1}^{\infty}\frac{(\ln(1+x) / (1+s))^2}{1 + |\ln(1+x) / (1+s)|}\ind_{\{|\ln(1 + x)| \leq  1\}}\nu_R(ds,dx) \\
& \leq \int_0^{\infty}\int_{-1}^{\infty}\frac{\ln(1+x)^2}{(1+s)^2}\ind_{\{|\ln(1 + x)| \leq  1\}}\nu_R(ds,dx) \\
& \leq \int_0^{\infty}\int_{-1}^{\infty}\frac{\ln(1+x)^2}{(1+s)^p}\ind_{\{|\ln(1 + x)| \leq  1\}} \nu_R(ds, dx) < \infty.
\end{split}
\end{equation*}

Then, note that by Young's inequality, $ab \leq \frac{a^n}{n} + \frac{b^m}{m}$, for $a,b > 0$, with $n = \frac{1}{p - 1}$ and $m$ given by $\frac{1}{n} + \frac{1}{m} = 1$, we obtain for all $s \geq 0$ and $x \in \mathbb{R}$,
$$\frac{1}{(1+s)+|\ln(1+x)|} \leq \frac{1}{n^{1/n}m^{1/m}(1+s)^{p-1}|\ln(1+x)|^{2-p}}.$$
Denoting $K = \frac{1}{n^{1/n}m^{1/m}}$ and since $\Delta H''_t =\ln(1 + \Delta R_t) \ind_{\{|\ln(1 + \Delta R_t)| > 1\}}$, we have
\begin{equation*}
\begin{split}
\tilde{Q}_{\infty} & = \int_0^{\infty}\int_{-1}^{\infty}\frac{(\ln(1+x) / (1+s))^2}{1 + |\ln(1+x) / (1+s)|}\ind_{\{|\ln(1 + x)| > 1\}}\nu_R(ds,dx) \\
& = \int_0^{\infty}\int_{-1}^{\infty}\frac{\ln(1+x)^2}{(1+s)(1 + s + |\ln(1+x)|)}\ind_{\{|\ln(1 + x)| > 1\}}\nu_R(ds,dx) \\
& \leq K \int_0^{\infty}\int_{-1}^{\infty}\frac{|\ln(1+x)|^p}{(1+s)^p}\ind_{\{|\ln(1 + x)| > 1\}}\nu_R(ds,dx) < +\infty.
\end{split}
\end{equation*}

Finally, to show that $\lim_{t \to \infty}\frac{R^c_t}{t} = 0$ $({\bf P}-a.s.)$, we apply Theorem 9 p.142-143 in \cite{LSh} again. Since $R^c$ is continuous, the theorem tells us that it is enough that $\int_0^{\infty}(1+s)^{-2}d \langle R^c \rangle_s < \infty$. But, this holds by assumption. 

So $\lim_{t \to \infty}\frac{\hat{R}_t}{t} = D$ and by Theorem \ref{t3}, if $D < 0$, we obtain for all $y > 0$ that $${\bf P}(\tau(y) < \infty) = 1.$$ 
\end{proof}

In the case when $R$ is a L\'{e}vy process, the assumptions in the proposition above simplify considerably and correspond to the conditions in \cite{P2} (under slightly different integrability assumptions).

\begin{corr}\label{coro_levyruin}
Suppose that $R$ is a L\'{e}vy process with triplet $(a_R, \sigma_R^2, \nu_R)$. Assume that $X_t = a_X t + \sigma_X W_t$, for all $t \geq 0$, with $a_X \leq 0$, $\sigma_X \geq 0$ and $a_X^2 + \sigma_X > 0$. Assume that there exists $p \in (1,2)$ such that
$$\int_{-1}^{\infty}|\ln(1+x)|^p \ind_{\{|\ln(1+x)| > 1\}} \nu_R(dx) < \infty.$$
In addition, assume that
$$a_R - \frac{1}{2}\sigma_R^2 + \int_{-1}^{\infty} \left(\ln(1+x) - x \ind_{\{|\ln(1+x)| \leq 1\}}\right)\nu_R(dx) < 0.$$
Then, for all $y > 0$,
$${\bf P}(\tau(y) < \infty) = 1.$$
\end{corr}

\begin{proof}
Since $R$ is a L\'{e}vy process, its semimartingale characteristics are given by $B_t = a_R t$, $\langle R^c \rangle_t = \sigma_R^2 t$ and $\nu_R(ds, dx) = \nu_R(dx)ds$ (see e.g. Corollary II.4.19, p.107, in \cite{JSh}). Note that since $R$ is a L\'{e}vy process, $\hat{R}$ is also a L\'{e}vy process, and
\begin{equation*}
\begin{split}
\int_{-1}^{\infty}|\ln(1+x)|^2 & \ind_{\{|\ln(1+x)| \leq 1\}}\nu_R(dx) \\
& = \int_0^1\int_{-1}^{\infty}|\ln(1+x)|^2\ind_{\{|\ln(1+x) \leq 1|\}} \nu_R(dx)ds \\
& = \E\left(\sum_{0 < s \leq 1}|\ln(1+\Delta R_s)|^2 \ind_{\{|\ln(1+\Delta R_s)| \leq 1\}}\right) \\
& = \E\left(\sum_{0 < s \leq 1}(\Delta \hat{R}_s)^2\ind_{\{|\Delta \hat{R}_s| \leq 1\}}\right) \\
& = \int_0^1\int_{\mathbb{R}}x^2\ind_{\{|x| \leq 1\}} \nu_{\hat{R}}(dx)ds < \infty,
\end{split}
\end{equation*}
so (ii) of Proposition \ref{prop_ruin1} holds. The conditions (i) and (iii) follow directly from the assumptions.
\end{proof}

\end{section}

\section*{Acknowledgements}
The authors would like to acknowledge financial support from the D\'{e}fiMaths project of the "F\'{e}d\'{e}ration de Recherche Math\'{e}matique des Pays de Loire" and from the PANORisk project of the "R\'{e}gion Pays de la Loire". We would also like to thank the French governement's "Investissement d'Avenir" program ANR-11-LABX-0020-01 for its stimulating mathematical research programs.

\end{document}